\documentclass[11pt]{amsart}

\usepackage{MnSymbol}
\usepackage{bbm, amsbsy}
\usepackage{psfrag}
\usepackage{hyperref}
\usepackage[capitalise,compress]{cleveref}

\newtheorem{theorem}{Theorem}[section]

\newtheorem{corollary}[theorem]{Corollary}

\theoremstyle{definition}

\theoremstyle{remark}
\newtheorem{remark}[theorem]{Remark}

\newcommand{\N}{{\mathbb N}}

\newcommand{\MM}{\mathbb{M}}
\newcommand{\HH}{\mathbb{H}}
\newcommand{\RR}{\mathbb{R}}
\newcommand{\SPH}{\mathbb{S}}
\newcommand{\dd}{\mathrm{d}}

\usepackage{comment}
\usepackage{tikz}
\usepackage{float}
\usepackage{relsize}
\usepackage{enumitem}
\usepackage[T1]{fontenc}
\usepackage[english]{babel}
\usepackage{color}
\usepackage{cite}
\usepackage{graphicx}
\usepackage{caption}
\usepackage{subcaption}
\usepackage{cite}

\date{\today}

\begin{document}

\sloppy
\title[Wicksell's corpuscle problem]{Wicksell's corpuscle problem\\ in spaces of constant curvature}

\author{Panagiotis Spanos}
\address{Panagiotis Spanos, Ruhr University Bochum, Faculty of Mathematics, Bochum, Germany}
\email{panagiotis.spanos\@@{}rub.de} 

\author{Christoph Th\"ale}
\address{Christoph Th\"ale, Ruhr University Bochum, Faculty of Mathematics, Bochum, Germany}
\email{christoph.thaele\@@{}rub.de} 

\date{}

\begin{abstract}
In this paper, we study Wicksell’s corpuscle problem in spaces of constant curvature, thus extending the classical Euclidean framework. We consider a particle process of balls with random radii in such a space, assumed to be invariant under the action of the full isometry group. We refer to the process of intersections of these balls with a fixed totally geodesic hypersurface as the induced process. We first derive a section formula expressing the distribution of the induced process in terms of that of the original process. Conversely, assuming that the distribution of the induced process is known, we establish an inversion formula that recovers the distribution of the original process. Finally, we show that in the limit as the curvature of the space tends to zero, our section and inversion formulas converge to the classical Euclidean results.
\end{abstract}

\thanks{This work has been supported by the German Research Foundation (DFG) via SPP 2265 Random Geometric Systems.}
\maketitle

\baselineskip=16pt

\section{Introduction}

Wicksell’s corpuscle problem, first formulated in 1925 by Wicksell~\cite{Wicksell1,Wicksell2},
asks how to reconstruct the radius distribution of a population of three-dimensional spherical particles from the observed distribution of radii of their {planar} circular
cross-sections.  In stereology this inverse task is known as the
{unfolding process}.  Although originally motivated by biological
questions, e.g.\ quantifying tissue structure in humans and other animals, the same mathematical
framework applies to any situation where only lower-dimensional slices of a
higher-dimensional structure are observable. Over the past century the problem has been analysed and extended in many directions by various
authors, see for example~\cite{Mecke1980,Dress1992,Ohser2001,Groeneboom1995,Gili2024}. For a survey on the topic
we refer to~\cite{Mecke1987}. Modern treatments typically
express the unfolding step as an Abel-type integral equation whose kernel
encodes the stereological geometry, and then solve it by regularised
deconvolution or non-parametric maximum likelihood estimation.  Recent work also explores Bayesian formulations that naturally incorporate measurement error and prior shape constraints.

Zähle~\cite{zahle} made an early attempt to extend Wicksell's corpuscle problem to
space of positive curvature, that is, to spherical spaces.  While the underlying ideas are conceptually sound, the
derivation contains an inaccuracy that leads to oversimplified and ultimately incorrect results, {see Remark \ref{rem:zahlecorrection} for a detailed explanation.}  In the present paper, we provide a fully rigorous formulation and
solution of Wicksell’s problem in all spaces of constant curvature, encompassing both
the {spherical} \textit{and} {hyperbolic} spaces simultaneously.  Moreover, we show that in the
limit as the curvature approaches zero, our generalized integral equations reduce
precisely to the classical Euclidean formulas. We thus provide the first complete treatment of Wicksell’s problem in both positively and negatively curved spaces, where curvature affects observable quantities in non-trivial ways. We also mention here that understanding Wicksell’s problem in curved spaces might have potential relevance in contexts where the ambient space is inherently non-Euclidean, such as in the modelling of biological tissues on curved surfaces, spatial networks embedded in hyperbolic space, or in imaging techniques that involve geodesic slices through non-flat structures.

The study of stochastic geometry models in spaces of constant curvature has recently attracted significant mathematical interest. In particular, models considered in hyperbolic space have attracted growing attention, as such settings can at times give rise to richer structural phenomena and novel asymptotic behaviours absent in the Euclidean case. For example, Boolean models, that is coverage processes formed by placing random “grains” at the points of a Poisson point process, have been formulated in hyperbolic space \cite{Bool1,Bool2,Bool3}.
Similarly, Poisson cylinder models, which consist of random systems of geodesic cylinders in spaces of constant negative curvature, exhibit intersection patterns and percolation properties that differ from those in the Euclidean setting \cite{Cyl1,Cyl2}. A range of other stochastic geometry models, including Poisson–Voronoi tessellations \cite{PV1,PV2}, their low-intensity limit \cite{AchilleEtAl}, Poisson hyperplane tessellations \cite{Hyp1,Hyp2,Hyp3,Hyp4}, random subspaces \cite{Hyp1,RSS}, percolation models \cite{perc}, random geometric graphs, complex networks and random connection models \cite{RGG1,RGG2,RGG3,RCM1,RCM2,RCM3}, have also been studied in both hyperbolic and spherical geometries. {In continuation of this line of research within stochastic geometry, we provide the solution of the corpuscle problem in the setting of such curved spaces.}

{The remainder of this paper is organized as follows. In Section \ref{sec:SetUpNotaiton}, we introduce the geometric framework and the stochastic process we will be working with. In Section \ref{sec:prelim}, we provide geometric formulas unified across all curvatures. In Section \ref{sec:form}, we derive the section formulas for the induced process in curved spaces, whereas in Section \ref{sec:inver}, we present the corresponding inversion formulas that solve Wicksell’s problem simultaneously in all constant curvature spaces. Finally, in Section \ref{sec:asympt}, we show that the formulas we obtain for curved spaces converge to those in the classical Euclidean case as the curvature tends to zero.}

\section{Set up and notation}\label{sec:SetUpNotaiton}

Let $\MM_k^d$ be the $d$-dimensional standard space of constant sectional curvature $k\in(-\infty,\infty)$. More precisely, for $k=0$ this space is the usual $d$-dimensional Euclidean space $\mathbb{R}^d$, for $k=1$ we obtain the $d$-dimensional sphere of radius $1$ and for $k=-1$ the $d$-dimensional hyperbolic space $\HH^d$. The space $\MM_k^d$ can be regarded as a Riemannian manifold, endowed with a Riemannian metric $\dd_k(\,\cdot\,,\,\cdot\,)$ and volume $\mu_{k,d}$. Specifically:
\begin{itemize}
\item \textit{Euclidean space:} If $k=0$ then $\MM_0^d=\RR^d$, $d_0(x,y):=\|x-y\|$ for $x,y\in\RR^d$ is the Euclidean distance and $\mu_{0,d}$ the $d$-dimensional Lebesgue measure.

\item \textit{Spherical spaces:} If $k>0$ then we can take for $\MM_k^d$ the $d$-dimensional sphere $$\SPH_k^d:=\{ x=(x_1,\ldots,x_{d+1}) \in \RR^{d+1} : x_1^2+\ldots+x_{d+1}^2=k^{-1} \} \subset\RR^{d+1},$$ which we think of being embedded into the Euclidean space $\RR^{d+1}$. For $\dd_k$ we take the  geodesic distance $$\dd_k(x,y):=k^{-1}\arccos(\langle x,y\rangle),\qquad x,y\in\SPH_k^d,$$ with $\langle\,\cdot\,,\,\cdot\,\rangle$ being the standard scalar product in $\RR^{d+1}$, finally $\mu_{k,d}$ is the usual spherical Lebesgue measure on $\SPH_k^d$.

\item \textit{Hyperbolic spaces:} If $k<0$ we let $$\MM_k^d=\HH_k^d:=\{x=(x_1,\ldots,x_{d+1})\in\RR^{d+1}:x_1^2+\ldots+x_d^2-x_{d+1}^2=k^{-1}, {x_{d+1}>0} \}\subset\RR^{d+1}$$ be a single-sheet hyperboloid in $\RR^{d+1}$. For $\dd_k$ we take the function $$\dd_k(x,y):=(-k)^{-1}{\rm arcosh}([x,y]),\qquad x,y\in\HH_k^d,$$ with the Lorenzian scalar product $[x,y]:=x_1y_1+\ldots+x_dy_d-x_{d+1}y_{d+1}$, $x=(x_1,\ldots,x_{d+1}),y=(y_1,\ldots,y_{d+1})\in\HH_k^d$. Moreover, we let $\mu_{k,d}$ be the Hausdorff measure on $\HH_k^d$ induced by the metric $\dd_k$.
\end{itemize}

Next, we let $\eta_d$ be a stationary point process on $\MM_k^d$ with intensity measure $N_d\mu_{k,d}$, where $N_d\in(0,\infty)$. Here, stationarity of $\eta_d$ refers to the property that the distribution of $\eta_d$ is invariant under all isometries of the space $\MM_k^d$.
Furthermore, we fix a probability measure $\nu_d$ on the positive half-axis $(0,\infty)$ and let $\xi_d$ be an independent $\nu_d$-marking of $\eta_d$. This means that each point of $\eta_d$ gets marked by an independent copy of a random variable with distribution $\nu_d$. Each point of the marked point process $\xi_d$ is a pair $(x,R)\in \MM_k^d\times (0,\infty)$, which geometrically corresponds to a random ball with centre $x\in\MM_k^d$ and radius $R\in(0,\infty)$. This way a random process of balls is created by $\xi_d$, which is also invariant under the isometries of the space $\MM_k^d$. The intensity measure of the random process $\xi_d$ is given by the product measure $$\Lambda_d:=N_d\,\mu_{k,d} \otimes \nu_d.$$

To proceed, consider a fixed $(d-1)$-dimensional totally geodesic submanifold $L\subset \MM_k^d$ and consider the intersection of the random ball process with $L$. In our model spaces $L$ is a usual affine hyperplane in $\RR^d$ if $k=0$, and otherwise $L$ can be realized as the intersection of $\SPH_k^d\subset\RR^{d+1}$ or $\HH_k^d\subset\RR^{d+1}$ with a linear subspace of $\RR^{d+1}$ of dimension $d$. Since any isometry of $\MM_k^d$ that fixes $L$ can be extended to the whole of $\MM^d_k$ and the random ball process we consider is invariant under all isometries of $\MM_k^d$, the induced random process of balls within $L$ is invariant under all isometries of $\MM_k^d$ that fix $L$. Therefore, the intersection of the random ball process in $\MM_k^d$ with $L$ induces a random ball process within $L$, which in turn can be described by a marked point process $\xi_{d-1}$ on $L$ with marks in $(0,\infty)$ whose intensity measure can be written as $$\Lambda_{d-1}=N_{d-1}\,\mu_{k,L}\otimes\nu_{d-1}.$$ Here, $N_{d-1}\in(0,\infty)$, $\mu_{k,L}$ is the canonical volume measure on $L$ (which may be identified with $\mu_{k,d-1}$) and $\nu_{d-1}$ is some probability measure on $(0,\infty)$.

\medspace

In the present paper we consider the following quantities and the relations between them:
\begin{itemize}
\item The intensity $N_d$, which has an interpretation as the mean number of ball-centres per unit $\mu_{k,d}$-volume in $\MM_k^d$.
\item The intensity $N_{d-1}$, which is the mean number of ball-centres per unit $\mu_{k,L}$-volume of the induced process within $L$.
\item The distribution function $D_d(r):=\nu_d((0,r])$ of the radii and the mean radius $\rho_d:=\int_0^\infty r\,\nu_d(\dd r)$ of the original random ball process in $\MM_k^d$.
\item The distribution function $D_{d-1}(r):=\nu_{d-1}((0,r])$ of the radii and the mean radius $\rho_{d-1}:=\int_0^\infty r\,\nu_{d-1}(\dd r)$ of the induced process within $L$.
\end{itemize}

We are interested in expressing the the probability measure $\nu_{d-1}$ through $\nu_d$ (intersection formula), and vice versa (inversion formula).

\section{Preliminaries}\label{sec:prelim}

In the proofs of our results it will turn out to be crucial to change coordinates of the space $\MM_k^d$, where $k\in\RR$. For this, we fix, as in the previous section, a totally geodesic hypersurface $L\subset\MM_k^d$. Then, each point $x\in\MM_k^d$ can be parametrized by its orthogonal projection $y:=p_L(x)$ onto $L$ and its geodesic distance  $\vert h \vert:=\mathrm{dist} ( x, L)$ from $L$, and a sign depending on its orientation. This parametrization is also called \textit{equidistant decomposition}, see \cite[Chapter 3.1]{Solodovnikov}. Through this diffeomorphism $x\mapsto\big(p_L(x),\mathrm{dist}(x,L)\big)$ the volume element $\mu_d(\dd x)$ can be expressed as
\begin{align}\label{eqdeco}
\mu_d(\dd x) = \cos^{d-1}(\sqrt{k} h) \, \dd h \,  \mu_{L}(\dd y),
\end{align}
where $\mu_{L}(\dd y)$ is the Riemannian volume element on $L$, see the table on p.\ 90 of \cite{Solodovnikov}. {Here and in what follows we use the convention that
\begin{align}\label{eq:CosConvention}
\cos(\sqrt{k}h) = \begin{cases}
	\cos(\sqrt{k}h)\quad & \text{ if } k>0\\
	1\quad & \text{ if } k=0\\
	\cosh(\sqrt{-k}h)\quad & \text{ if } k<0.\\
\end{cases}
\end{align}}
We note that if $k>0$, then $h$ can only take values in the bounded interval $[-\frac{\pi}{2\sqrt{k}},\frac{\pi}{2\sqrt{k}}]$, whereas $h\in\RR$ for all other cases. In order to provide formulas for all curvature parameters simultaneously we will encode this behaviour by defining
\[l_k := 
\begin{cases}
\frac{\pi}{2 \sqrt{k}} \quad & \text{ if } k>0
\\
\infty \quad & \text{ if } k\leq 0.
\end{cases}
\]
In this paper we will often encounter intervals like $(r,l_k]$ or $[r,l_k]$ for $r\geq 0$. To deal with all curvatures simultaneously, we will interpret these as
$$
(r,l_k] := \begin{cases}
\{x\in\RR:r<x\leq l_k\}\quad & \text{ if } k>0\\
\{x\in\RR:r<x< l_k\} \quad & \text{ if } k\leq 0,
\end{cases}\qquad
[r,l_k] := \begin{cases}
	\{x\in\RR:r\leq x\leq l_k\} \quad & \text{ if } k>0\\
	\{x\in\RR:r\leq x < l_k\} \quad & \text{ if } k\leq 0.
\end{cases}
$$

Using our convention \eqref{eq:CosConvention}, we {formulate the \textit{Pythagorean theorem for spaces of constant curvature $k$}. It says that
\begin{align}\label{lawcos}
\cos\left( \sqrt{k} \dd_{k}(A,C) \right) = \cos \left( \sqrt{k}  \dd_k(A,B) \right) \cos\left( \sqrt{k}  \dd_k (B,C)\right),
\end{align}
for any three points $A,B,C\in\MM_k^d$ which are such that the geodesics $AB$ and $BC$ meet at a right angle.}

For $k>0$, $\MM_k^d$ is a sphere of radius $R=k^{-\frac{ 1}{2}}$. The formula arises from the \textit{spherical law of cosines} for a sphere of radius $R$. Let $O$ be the centre of this sphere and let $A, B, C$ be three points such that the angle between the arcs $AB$ and $BC$ is a right angle. If $a,b$ and $c$ are the angles of $OA$ with $OB$, $OB$ with $OC$ and $OA$ with $OC$, respectively, then 
$
\cos   a = \cos   b \cos   c
$.
In spherical distances this translates to \ref{lawcos}. 

For $k<0$, where $\MM_k^d$ is a scaled version of the standard hyperbolic space of constant curvature $-1$, there exists a corresponding relation called the \textit{first law of cosines}, see \cite[Theorem 3.5.3]{Ratcc}. Let $A,B,C$ be three points in $\MM_k^d$ such that the angle between the geodesics $AB$ and $BC$ is a right angle. If we rescale $\MM_k^d$ by $\sqrt{-k}$, then the three points correspond to $a,b,c\in\HH^d$ and the angle between $ab$ and $bc$ remains a right angle. The first law of cosines is 
$
\cosh \left(\dd_{-1}(a,c) \right)=\cosh\left( \dd_{-1}(a,b)\right) \cosh\left( \dd_{-1} (b,c)\right)
$
{which translates to \ref{lawcos}.}

\begin{remark}\label{rem:zahlecorrection}  
The study by Z{\"a}hle \cite{zahle} makes a valuable contribution to Wicksell’s problem in spherical spaces. However, as already anticipated in the introduction we have identified an oversight in the expression for the volume element under equidistant decomposition: the exponent \(d-1\) on the cosine factor in \eqref{eqdeco} is inadvertently omitted in \cite[Equation (2.1)]{zahle}. While this does not affect the two–dimensional case, it does lead to substantial deviations in dimensions \(d\ge3\). In what follows, we present the necessary adjustments and extend the original results to all \(d\ge2\) as well as to negatively curved spaces.
\end{remark}

\section{Section formulas}\label{sec:form}

Recall the setup and the notation introduced in Section \ref{sec:SetUpNotaiton}. The goal of this section is to derive section formulas which express the parameters $N_{d-1}$ and $\nu_{d-1}$ of the induced process within $L$ in terms of the parameters $N_d$ and $\nu_d$ of the original random ball process. For that purpose, let $x\in\MM_k^d$ be a point of the marked point process $\xi_d$ and $R_x$ be the corresponding radius. We represent $x$ as $(y,h)$ using the equidistant decomposition as introduced in Section \ref{sec:prelim}. If the corresponding radius $R_x$ is larger than $h$, then the ball $B(x,R_x)$ intersects $L$ and induces a ball $B_L(y,R_y)$ within $L$. 

If $k\neq 0$, for any point $z$ on the boundary of $B_L(y,R_y)$ by Equation \eqref{lawcos} and using that
$\dd_{k}(x,z)=R_x
$,
$
\dd_k(x,y)= \vert h \vert   
$
and
$
\dd_k (y,z)=R_y
$
we obtain
$
\cos\left( \sqrt{k} R_x \right) = \cos \left( \sqrt{k}    h     \right) \cos\left( \sqrt{k} R_y \right)
$, where we recall our convention \eqref{eq:CosConvention}.
Thus,
$$
\cos(\sqrt{k}R_y)=\frac{\cos(\sqrt{k}R_x)}{\cos(\sqrt{k}   h     )}.
$$
If $k=0$, then $R_y=\sqrt{R^2_x-h^2}$. 
To summarize, we define on $\lbrace (t,h)\in\mathbb{R}^2: 0\leq \vert h \vert  \leq t \leq l_k \rbrace$,
\[
\alpha_k(t,h) :=
\begin{cases}
\frac{1}{\sqrt{-k}}\mathrm{arcosh}\left(\frac{\cosh(\sqrt{-k} t)}{\cosh(\sqrt{-k}   h    )}\right)&\quad \text{if } k<0\\
\sqrt{t^2-h^2}&\quad \text{if } k=0\\
\frac{1}{\sqrt{ k}}\arccos\left(\frac{\cos(\sqrt{k} t)}{\cos(\sqrt{k}    h     )}\right)&\quad \text{if } k>0.
\end{cases}
\]
Then, whenever $R_x\geq h$, we have $R_y=\alpha_k(R_x,h)$.  For fixed $h$, we may treat $\alpha_k$ as a bijective function of $t$ only, which admits an inverse. So for fixed $h$, we define $\alpha^{-1}_k(t):=\beta_k(t,h)$ by
\[
\beta_k(t,h) :=
\begin{cases}
\frac{1}{\sqrt{-k}}\mathrm{arcosh}\left( \cosh(\sqrt{-k} t) \cosh(\sqrt{-k}    h     ) \right)&\quad \text{if } k<0\\
\sqrt{t^2+h^2}&\quad \text{if } k=0\\
\frac{1}{\sqrt{ k}}\arccos\left( \cos(\sqrt{k} t) \cos(\sqrt{k}    h     ) \right)&\quad \text{if } k>0.
\end{cases}
\]
It follows that $R_x=\beta_k(R_y,h)$. Note that no restriction is needed here, as $R_x$ is always larger than $h$. 
Observe that for any $k$, as $t$ tends to $0$, we have that $\beta_k(t,h)$ tends to $h$. This means that as the radius $R_y$ tends to zero, the distance of $x$ from $y$ tends to the radius $R_x$. 

The next result, a section formula for constant curvature spaces, is our first contribution. A special case is illustrated in Figure \ref{fig:sectional_densities}.

\begin{figure}[t]
\centering
\includegraphics[width=0.7\columnwidth]{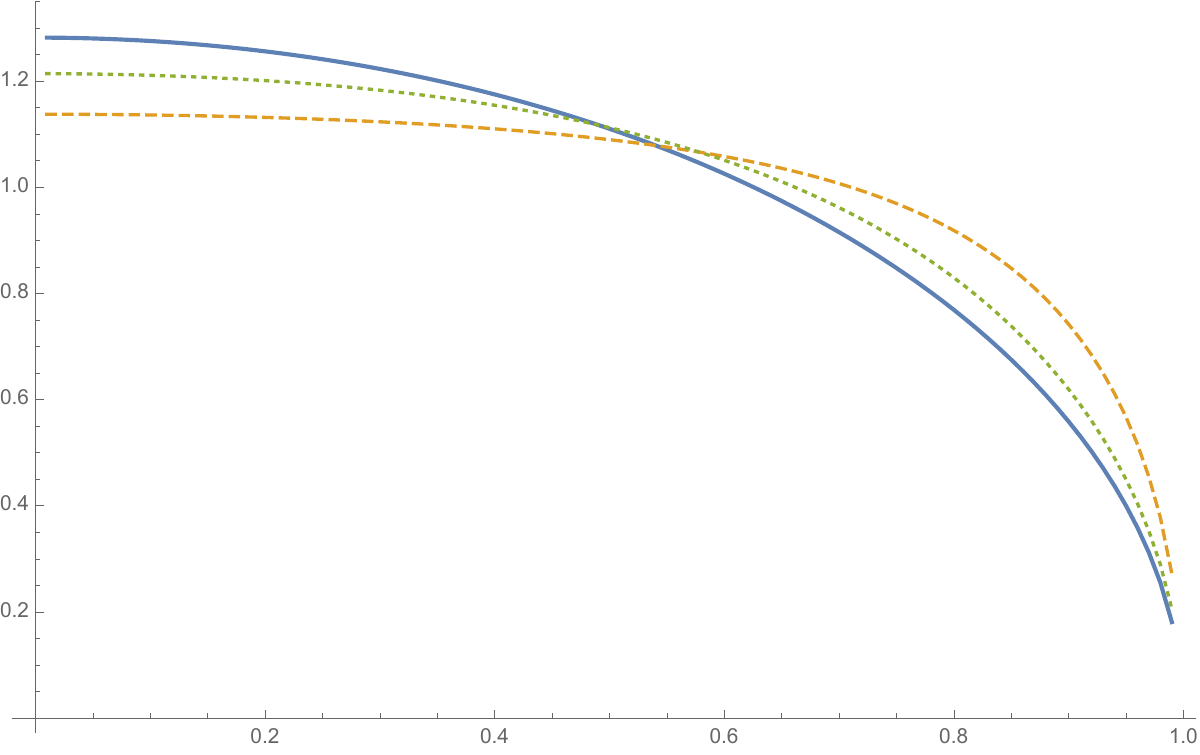}
\caption{Densities on $[0,1]$ of the measures $\nu_{d-1}$ of observed section radii when slicing a random process of balls of fixed radius \(1\) in a $3$-dimensional spaces of constant curvature $k$. The curves correspond to \(k=0\) (solid, blue), \(k=+1\) (dashed, orange), and \(k=-1\) (dotted, green).}
\label{fig:sectional_densities}
\end{figure}

\begin{theorem}[Section formula for constant curvature spaces]\label{th:ratioestimate}
Let $\MM_k^d$ be a space of constant curvature $k\neq 0$, let $L$ be a totally geodesic hypersurface and recall that $l_k=\frac{\pi}{2\sqrt{k}}$ when $k>0$, and $l_k=\infty$ otherwise. Then,
$$
N_{d-1}   \nu_{d-1}((r,l_k]) 
= N_d \int_r^{l_k} \int_{-\alpha_k(R,r)}^{ \alpha_k(R,r)}  \cos^{d-1}(\sqrt{k}h)  \,\dd h\, \nu_{d}(\dd R),\qquad r\in[0,l_k].
$$
Specifically, for the ratio of the intensity parameters, we obtain
$$
\frac{N_{d-1}}{N_d}=\int_0^{l_k} \int_{-R}^R \cos^{d-1}(\sqrt{k} h) \, \dd h \, \nu_d(\dd R).
$$
\end{theorem}
\begin{proof}
We fix a set $F=C\times (r,l_k]\subset L\times (0,l_k]$, where $C$ a is compact subset of $L$. We compute the value of $\Lambda_{d-1}(F)$ in two different ways:
\begin{align}\label{eq:Lambda}
\Lambda_{d-1}(F)=N_{d-1} \mu_{k,d-1}(C)\nu_{d-1}((r,l_k]) .
\end{align}
A point $(y,R_y)$ of $\xi_{d-1}$ arises from a point $(x,R_x)$ of $\xi_d$ with $x=(y,h)$ if and only if $h\leq R_x$. Therefore, the expected number of marked points of $\xi_{d-1}$ in $F$ equals the expected number of marked points of $\xi_d$ for which the corresponding radius is larger than their distance from the set $C$:
\begin{align}\label{eq:xtoy}
\mathbb{E}\sum_{(y,R_y)\in\xi_{d-1}} \mathbf{1}_F(y,R_y)
=  
\mathbb{E} \sum_{\substack{((y,h),R_x)\in \xi_d \\ h\leq R_x}} \mathbf{1}_F(y,\alpha_k(R_x,h)).
\end{align} 
Applying the Campbell theorem for marked point processes, see \cite[Section 4.2.2]{Mecke1987}, Fubini's theorem for the representation $\Lambda_d=N_d\mu_{k,d} \otimes \nu_d$ and the equidistant decomposition \eqref{eqdeco} we obtain 
\begin{align}
&\mathbb{E} \sum_{\substack{((y,h),R_x)\in \xi_d \\ h\leq R_x}} \mathbf{1}\{ F \}(y,\alpha_k(R_x,h))\\
&=
\int_{\MM_k^d \times (0,\infty)} \mathbf{1}_F(y, \alpha_k(R,h)) \, \Lambda_d(\dd (x,R))\nonumber \\
&=N_d
\int_{\MM_k^d \times (0,\infty)} \mathbf{1}_F(y, \alpha_k(R,h)) \, \mu_d (\dd x) \, \nu_{d}(\dd R)\nonumber \\
&=N_d\int_0^{l_k} \int_L \int_{-R}^R \mathbf{1}_F(y,\alpha_k(R,h)) \cos^{d-1}(\sqrt{k}h)  \,\dd h\,  \mu_L (\dd y) \, \nu_{d}(\dd R)\nonumber \\
&= N_d\mu_{d-1}(C)\int_0^{l_k}  \int_{-R}^R  \mathbf{1}_{ (r,l_k]}(\alpha_k(R,h)) \cos^{d-1}(\sqrt{k}h)  \,\dd h\, \nu_{d}(\dd R). \nonumber  
\end{align}
Observe that $ \mathbf{1}_{ (r,l_k]}(\alpha_k(R,h))$ is non-zero if and only if $r < \alpha_k(R,h)\leq l_k$, {which provides for $h$ the following inequality: $  \alpha_k(R,r)> \vert h\vert  > 0$ and $R > r$.} Thus, 
 \begin{align}\label{eq:firstpartmean}
 \mathbb{E} \sum_{\substack{((y,h),R_x)\in \xi_d \\ h\leq R_x}} \mathbf{1}_F(y,\alpha_k(R_x,h))
&= N_d\mu_{d-1}(C)\int_r^{l_k}  \int_{-\alpha_k(R,r)}^{ \alpha_k(R,r)}   \cos^{d-1}(\sqrt{k}h)  \,\dd h\, \nu_{d}(\dd R).
 \end{align}
Combining \eqref{eq:firstpartmean}, \eqref{eq:xtoy} and \eqref{eq:Lambda}, we obtain
$$
N_{d-1}   \nu_{d-1}((r,l_k]) 
= N_d \int_r^{l_k} \int_{-\alpha_k(R,r)}^{ \alpha_k(R,r)}  \cos^{d-1}(\sqrt{k}h)  \,\dd h\, \nu_{d}(\dd R).
$$
Taking $r=0$ in the definition of $F$ we finally conclude that
$$
\frac{N_{d-1}}{N_d}=  \int_0^{l_k}  \int_{-R}^R   \cos^{d-1}(\sqrt{k}h)  \,\dd h\, \nu_{d}(\dd R),
$$
which completes the proof.
\end{proof}

\begin{remark}\label{re:2}
\begin{itemize}
\item[(a)] In case that $k>0$, making the substitution $u=\sin^2(\sqrt{k} h)$, we obtain the following expression for the ratio of intensities: 
\begin{align*}
\frac{N_{d-1}}{N_d} &=\int_0^{\frac{\pi}{2 \sqrt{k}}} \frac{1}{\sqrt{k}} \int_0^{\sin^2(\sqrt{k}R)} (1-u)^{\frac{d-3}{2}} u^{-\frac{1}{2}}\, \dd u \, \nu_d(\dd R)\\
&=\int_0^{\frac{\pi}{2 \sqrt{k}}} \frac{1}{\sqrt{k}} B_{\sin^2(\sqrt{k} R)} \left(\frac{1}{2},\frac{d-1}{2}\right)  \, \nu_d(\dd R),
\end{align*}
where $B_y(a,b):=\int_0^y u^{a-1}(1-u)^{b-1}\,\dd u$ is the incomplete beta function, $a,b>0$, $0<y<1$. For $d=2$ this reduces to $\frac{N_{d-1}}{N_d}={2\over\sqrt{k}}\int_0^{\frac{\pi}{2\sqrt{k}}} \sin( \sqrt{k} R ) \,\nu_d(\dd R)$. It had been claimed in \cite[Equation (2.5)]{zahle} that the same reduction is possible for all space dimensions. However, this is not correct, as for $d=4$, for example, one has that $\frac{N_{d-1}}{N_d}={2\over 3\sqrt{k}}\int_0^{\frac{\pi}{2\sqrt{k}}} \sin(\sqrt{k}R)(2+\cos^2(2\sqrt{k}R))\,\nu_d(\dd R)$. 

\item[(b)] In case that $k<0$, by the binomial expansion, we obtain that the ratio equals 
$$
\frac{N_{d-1}}{N_d}=\int_0^{\infty} \frac{1}{2^{d-2}\sqrt{-k}}\sum_{i=0}^{d-1} 
\binom{d-1}{i}\frac{e^{(d-2i-1)\sqrt{-k} R}-1}{d-2i-1}
\, \nu_d(\dd R).
$$
\end{itemize}
\end{remark}

\begin{corollary}
The result of Theorem \ref{th:ratioestimate} can also be expressed by saying that for every non-negative measurable function $f:[0,l_k]\to\RR$,
\begin{align}
	\notag\int_0^{l_k} f(R)  \, \nu_{d-1} (\dd R) = &\left( \int_0^{l_k} \int_{-R}^R \cos^{d-1}(\sqrt{k} h) \, \dd h \, \nu_d(\dd R) \right)^{-1}\\
	& \qquad\times \int_0^{l_k} \int_{-R}^R f(\alpha_k(R,h))  \cos^{d-1}(\sqrt{k}h)  \,\dd h\, \nu_{d}(\dd R) .\label{eq:sectionformula}
\end{align}
\end{corollary}

\section{Inversion formulas}\label{sec:inver}

In this section we provide the inversion formula and thus the solution of  Wicksell's problem for all $d$-dimensional spaces of constant curvature. The idea is to choose a suitable function $f$ in \eqref{eq:sectionformula}, which essentially corresponds to an Abel-type inversion formula, similar as in the classical Euclidean case treated in \cite{Mecke1987}. We consider the cases $k<0$ and $k>0$ separately. In fact, the limits of the integrals defined by the hyperbolic functions differ in nature from the bounded trigonometric functions used in spherical spaces and we were not able to handle both case in our arguments in a unified way. 

\subsection{The case of negative curvature}

We start with the case of negative curvatures.

\begin{theorem}[Inversion formula for spaces of negative curvature]\label{thm:Inversionk<0}
Let $\MM_k^d$ be a space of constant curvature $k<0$. Then for any $a>0$,
\begin{align*}
\nu_d((a,\infty))&=\frac{N_{d-1} \sqrt{-k}  }{N_d \pi}\left( \frac{1}{\cosh^{d-2}(\sqrt{-k}a)  } \int_{a}^\infty \frac{\cosh^{d-1} (\sqrt{-k}y)}{ \sqrt{\cosh^2  (\sqrt{-k}y)-\cosh^2(\sqrt{-k}a)}} \,  \nu_{d-1}(\dd y)\right. \nonumber \\
&\qquad +
\left.
(d-1)\int_a^{\infty}\int_{\arcsin \left(\frac{\cosh(\sqrt{-k}a)}{\cosh(\sqrt{-k}y)} \right)}^{\frac{\pi}{2} }
    ( \sin\theta)^{-(d-1)}  \, \dd \theta \, \nu_{d-1}(\dd y) \right).
\end{align*}
\end{theorem}
\begin{proof}
Let $\tilde{\nu}_{d-1}$ denote the pushforward on $[1,\infty)$ of the measure $\nu_{d-1}$ on $[0,\infty)$ through the function $\cosh(\sqrt{-k}x)$.
We consider the function {$f(R):= \tilde{f}(\cosh (\sqrt{-k} R) )$,} where the precise definition of $\tilde{f}$ will be given later.
According to Equation \eqref{eq:sectionformula} {and the definition of $\alpha_k$,}
\begin{align}
I:=\int_1^{\infty}\tilde{f}(y)\,\tilde{\nu}_{d-1}( \dd y)  &=\int_0^{\infty} \tilde{f} (\cosh (\sqrt{-k} R) )\, \nu_{d-1}(\dd R)\nonumber \\
&=
\frac{N_d}{N_{d-1}} \int_0^{\infty} \int_{-R}^{R} \tilde{f}(\cosh (\sqrt{-k} \alpha_k(R,h)) \cosh^{d-1}(\sqrt{-k} h) \, \dd h \, \nu_d(\dd R)\nonumber \\
&=
\frac{N_d}{N_{d-1}} \int_0^{\infty} \int_{-R}^{R} \tilde{f}\left(\frac{\cosh (\sqrt{-k} R)}{\cosh (\sqrt{-k} h)}\right) \cosh^{d-1}(\sqrt{-k} h) \, \dd h \, \nu_d(\dd R)\nonumber \\
&=
\frac{N_d}{N_{d-1}} \int_0^{\infty} 2 \int_{0}^{R} \tilde{f}\left(\frac{\cosh (\sqrt{-k} R)}{\cosh (\sqrt{-k} h)}\right) \cosh^{d-1}(\sqrt{-k} h) \, \dd h \, \nu_d(\dd R).\label{eq:sectionform}
\end{align}
Here we {also} used Equation \eqref{lawcos} and the symmetry of the hyperbolic cosine function.  

We now set $u:=\cosh (\sqrt{-k}h)$. Then, $\dd u=\sqrt{-k} \sinh (\sqrt{-k}h) \,\dd h$, and 
$\sinh (\sqrt{-k}h) = ( \cosh^2 (\sqrt{-k}h )-1 )^{\frac{1}{2}}=(u^2-1)^{\frac{1}{2}} $. Therefore
$$
\frac{1}{\sqrt{-k}}( u^2-1)^{-\frac{1}{2}}\,\dd u =\dd h
$$
and so
$$
I=\frac{N_d}{N_{d-1}} \int_0^{\infty} \frac{2}{\sqrt{-k}} \int_{1}^{\cosh (\sqrt{-k} R)} \tilde{f}\left(\frac{\cosh (\sqrt{-k} R)}{u}\right)u^{d-1}(u^2-1)^{-\frac{1}{2}}  \, \dd u \, \nu_d(\dd R).
$$
Setting $z:=\cosh (\sqrt{-k}R)$, and noting that $\tilde{\nu_d}(\dd z) = \nu_d (\dd R)$, we arrive at 
\begin{align}\label{eq:zu}
I=\frac{N_d}{N_{d-1}} \frac{2}{\sqrt{-k}} \int_1^{\infty}  \int_{1}^{z} \tilde{f}\left(\frac{z}{u}\right)u^{d-1}(u^2-1)^{-\frac{1}{2}}  \, \dd u \, \tilde{\nu}_d(\dd z).
\end{align}
Now, choose $\tilde{f}(x):=\mathbf{1}_{(a,\infty)}(x) x^{d-1}(x^2-a^2)^{-\frac{1}{2}}$. This choice yields an Abel-type inversion formula, as in the classical Wicksell's corpuscle problem in Euclidean space \cite{Mecke1987}. 
More precisely, with this choice of $\tilde{f}$, \eqref{eq:zu} can be written as
\begin{align*}
I=\frac{N_d}{N_{d-1}} \frac{2}{\sqrt{-k}} \int_1^{\infty}  \int_{1}^{z} 
 \mathbf{1}_{(a,\infty)}\left(\frac{z}{u}\right) 
\frac{z^{d-1}}{u^{d-1}}(z^2u^{-2} - a^2)^{-\frac{1}{2}}
u^{d-1}(u^2-1)^{-\frac{1}{2}}
\, \dd u \, \tilde{\nu}_d(\dd z) .
\end{align*}
We incorporate the indicator function into the integration limits, noting that $\frac{z}{u}\geq a$ and $u\geq 1$, which is equivalent to $\frac{z}{a}\geq u \geq 1$. This yields
\begin{align*}
I&= \frac{N_d}{N_{d-1}} \frac{2}{\sqrt{-k}} \int_a^{\infty} z^{d-1} \int_1^{\frac{z}{a}}
(z^2u^{-2} - a^2)^{-\frac{1}{2}} (u^2-1)^{-\frac{1}{2}}
\, \dd u \, \tilde{\nu}_d(\dd z)\\
&=  \frac{N_d}{N_{d-1}} \frac{2}{\sqrt{-k}} \int_a^{\infty} z^{d-1} \int_1^{\frac{z}{a}}
ua^{-1}(z^2a^{-2} - u^2)^{-\frac{1}{2}} (u^2-1)^{-\frac{1}{2}}
\, \dd u \, \tilde{\nu}_d(\dd z)\\
&=  \frac{N_d}{N_{d-1}} \frac{\pi}{\sqrt{-k}} \int_a^{\infty} z^{d-1}a^{-1} \, \tilde{\nu}_d(\dd z).
\end{align*}
Combining the facts that $I=\int_0^{\infty}\tilde{f}(y)\,\tilde{\nu_{d-1}}( \dd y) $ and that $\tilde{f}(x)=\mathbf{1}_{(a,\infty)}(x) x^{d-1}(x^2-a^2)^{-\frac{1}{2}}$ together with the above calculation, we obtain
\begin{align}\label{eq:ndtond-1}
\int_a^\infty   x^{d-1} (x^2-a^2)^{-\frac{1}{2}} \, \tilde{\nu}_{d-1}(\dd x) = \frac{N_d}{N_{d-1}} \frac{\pi}{\sqrt{-k}} \int_a^{\infty} z^{d-1}a^{-1}  \, \tilde{\nu}_d(\dd z).
\end{align}

We proceed by considering  the tail $\tilde{\nu}_{d}((a,\infty))$ of the measure $\tilde{\nu}_{d}$: 
\begin{align*}
\tilde{\nu}_{d}((a,\infty))&= \int_a^{\infty} \, \tilde{\nu}_{d}(\dd z)\\
&=  \int_a^{\infty} \left( -
\int_z^\infty {\frac{\dd }{\dd u} \left( \frac{1}{u^{d-1}} \right) } \, \dd u\right)
z^{d-1} 
\, \tilde{\nu}_{d}(\dd z) \\
&= \int_a^{\infty} \int_a^{  u} \frac{d-1}{u^d}  z^{d-1} \, \tilde{\nu}_d(\dd z) \, \dd u\\
&=\int_a^{\infty} \left( \int_a^{  \infty}   z^{d-1} \, \tilde{\nu}_d(\dd z) -\int_u^{  \infty}  z^{d-1} \, \tilde{\nu}_d(\dd z) \right) \frac{d-1}{u^d} \, \dd u.
\end{align*}
We can now use \eqref{eq:ndtond-1} to see that
\begin{align}\label{eq:last1}
\tilde{\nu}_{d}((a,\infty) ) = \frac{N_{d-1} \sqrt{-k} }{N_d \pi} \int_a^\infty  \left(   \int_a^\infty \right.  &a x^{d-1} (x^2-a^2)^{-\frac{1}{2}} \, \tilde{\nu}_{d-1}(\dd x) \nonumber \\
&    \left.  -\int_u^\infty  u  x^{d-1} (x^2-u^2)^{-\frac{1}{2}} \, \tilde{\nu}_{d-1}(\dd x) \right)
\, \frac{d-1}{u^d} \, \dd u.
\end{align}
Evaluating the two integrals yields
\begin{align}\label{eq:last2}
\int_a^\infty   \int_a^\infty  \frac{d-1}{u^d} \, \dd u    \, \frac{a x^{d-1}}{ \sqrt{x^2-a^2}} \, \tilde{\nu}_{d-1}(\dd x) = \int_a^\infty a^{-(d-1)} \frac{a x^{d-1}}{ \sqrt{x^2-a^2}} \, \tilde{\nu}_{d-1}(\dd x)
\end{align}
and 
\begin{align}\label{eq:last3}
\int_a^\infty   \int_u^\infty   u x^{d-1} (x^2-u^2)^{-\frac{1}{2}} \, \tilde{\nu}_{d-1}(\dd x) \frac{d-1}{u^d} \, \dd u  &=(d-1) \int_a^\infty  x^{d-1}  \int_a^x  \frac{u^{-(d-1)}}{\sqrt{x^2-u^2}} \, \dd u\,   \tilde{\nu}_{d-1}(\dd x) .
\end{align}
Noting that $\nu_d((a,\infty))=\tilde{\nu}_d((\cosh(\sqrt{-k}a),\infty))$, from Equations \eqref{eq:last1}, \eqref{eq:last2} and \eqref{eq:last3} we obtain
\begin{align}\label{eq:lastform2}
\nu_d((a,\infty))=& \frac{N_{d-1} \sqrt{-k}  }{N_d \pi} \Bigg(  \frac{1}{\cosh^{ d-2 }(\sqrt{-k}a)}  \int_{\cosh(\sqrt{-k}a)}^\infty \frac{x^{d-1}}{\sqrt{x^2-\cosh^2(\sqrt{-k}a)}} \, \tilde{\nu}_{d-1}(\dd x)  \nonumber\\
 &\quad  +(d-1)\int_{\cosh(\sqrt{-k}a)}^\infty   \int_{\cosh(\sqrt{-k}a)}^x  \frac{u^{-(d-1)}}{\sqrt{x^2-u^2}} \, \dd u\, x^{d-1} \, \tilde{\nu}_{d-1}(\dd x)\Bigg)\nonumber\\
 =&  \frac{N_{d-1} \sqrt{-k}  }{N_d \pi}\Bigg( \frac{1}{\cosh^{d-2}(\sqrt{-k}a)  } \int_{a}^\infty \frac{\cosh^{d-1} (\sqrt{-k}y)}{ \sqrt{\cosh^2  (\sqrt{-k}y)-\cosh^2(\sqrt{-k}a)}} \,  \nu_{d-1}(\dd y) \nonumber \\
 &\quad  + (d-1)\int_{a}^\infty   \int_{\cosh(\sqrt{-k}a)}^{\cosh  (\sqrt{-k}y) } \frac{u^{-(d-1)}}{\sqrt{\cosh^2  (\sqrt{-k}y)-u^2} } \, \dd u 
 \, \cosh^{d-1}  (\sqrt{-k}y) \,  \nu_{d-1}(\dd y)\Bigg).
\end{align}
Now set $u:=\cosh (\sqrt{k} y) \sin \theta$, {for $\theta \in (\arcsin \left(\frac{\cosh(\sqrt{-k}a)}{\cosh(\sqrt{-k}y)} \right) , \frac{\pi}{2})$ }  Then $\dd u = \cosh (\sqrt{k} y) \cos \theta \, \dd \theta $ and 
$\cosh^2 (\sqrt{k} y) - u^2= \cosh^2 (\sqrt{k} y)\cos^2\theta$, and it follows that
\begin{align}\label{eq:sec3}
 \notag &\int_{\cosh(\sqrt{-k}a)}^{\cosh  (\sqrt{-k}y) } \frac{u^{-(d-1)}}{\sqrt{\cosh^2  (\sqrt{-k}y)-u^2} } \, \dd u \\
 &\qquad=
 \int_{\arcsin \left(\frac{\cosh(\sqrt{-k}a)}{\cosh(\sqrt{-k}y)} \right)}^{\frac{\pi}{2} }
  \frac{(\cosh (\sqrt{ k}y)\sin\theta)^{-(d-1)}}{\cosh (\sqrt{ k}y)\cos\theta} \cosh (\sqrt{ k}y)\cos\theta \, \dd \theta.
\end{align}
Substituting \eqref{eq:sec3} into \eqref{eq:lastform2}, we finally obtain
\begin{align*}
\nu_d((a,\infty))&=\frac{N_{d-1} \sqrt{-k}  }{N_d \pi}\Bigg( \frac{1}{\cosh^{d-2}(\sqrt{-k}a)  } \int_{a}^\infty \frac{\cosh^{d-1} (\sqrt{-k}y)}{ \sqrt{\cosh^2  (\sqrt{-k}y)-\cosh^2(\sqrt{-k}a)}} \,  \nu_{d-1}(\dd y) \nonumber \\
&\qquad +
(d-1)\int_a^{\infty}\int_{\arcsin \left(\frac{\cosh(\sqrt{-k}a)}{\cosh(\sqrt{-k}y)} \right)}^{\frac{\pi}{2} }
    ( \sin\theta)^{-(d-1)}  \, \dd \theta \, \nu_{d-1}(\dd y) \Bigg).
\end{align*}
This is the desired inversion formula. 
\end{proof}

\begin{remark}
In Theorem \ref{thm:Inversionk<0} we implicitly assumed that the probability measure $\nu_{d-1}$ arises as the radius distribution of the induced process of some random ball process in $\MM_k^d$. Not every probability measure on $(0,\infty)$ is necessarily of that form. In particular, if the integrals on the right-hand side are not well defined, $\nu_{d-1}$ is not of this form. For example, finiteness of the first integral assumes that $e^{(d-2)y}$ is integrable with respect to $\nu_{d-1}$. Such an effect does also appear in the flat case, but not for positive curvatures $k>0$, because all integrals are finitely supported.
\end{remark}

\subsection{The case of positive curvature}

Next, we turn to the case of positive curvatures. As discussed in Remark \ref{re:2} (a), this has incorrectly been treated in \cite{zahle}. For completeness, we present the correct inversion formula whose proof parallels that of Theorem \ref{thm:Inversionk<0}.

\begin{theorem}[Inversion formula for spaces of positive curvature]\label{th:inveposi}
Let $\MM_k^d$ be a space of constant curvature $k>0$. Then, for any $a\in(0,l_k)$,
\begin{align*}
\nu_d((a,l_k]) 
=&\frac{N_{d-1}\sqrt{k} }{N_d\pi}  
\left(
\frac{ 1}{\cos^{d-2} (\sqrt{ k}a)}   \int_a^{l_k}   \frac{\cos^{d-1} (\sqrt{ k}y)}{\sqrt{\cos^{ 2} (\sqrt{ k}a)-\cos^{ 2} (\sqrt{ k}y)}} \,  \nu_{d-1}(\dd y) 
\right.
 \\
&\quad +
\left.
(d-1)\int_a^{l_k} \int_0^{\arccos \frac{\cos (\sqrt{ k}y)}{\cos (\sqrt{ k}a)}}
\cos^{d-2}\theta\, \dd \theta \, \nu_{d-1}(\dd y) \right).
\end{align*}
\end{theorem}
\begin{proof}
We consider the pushforward measure $\tilde{\nu}$ of $\nu_{d-1}$ through the function $\cos (\sqrt{k}x)$ and define the function {$f(R):= \tilde{f}(\cos (\sqrt{k} R) )$,} which will be made precise later.
According to \eqref{eq:sectionformula}, {combined with the definition of $\alpha_k$ and \eqref{lawcos} we have}
\begin{align}
I=\int_1^{l_k}\tilde{f}(y)\,\tilde{\nu}_{d-1}( \dd y)  =
\frac{N_d}{N_{d-1}} \int_0^{l_k} 2 \int_{0}^{R} \tilde{f}\left(\frac{\cos (\sqrt{k} R)}{\cos (\sqrt{k} h)}\right) \cos^{d-1}(\sqrt{k} h) \, \dd h \, \nu_d(\dd R).\label{eq:sectionform2}
\end{align}
We set $u:=\sin (\sqrt{k}h)$ and $z=\cos (\sqrt{k} R)$, so that $\cos(\sqrt{k}h)=\sqrt{1-\sin^2(\sqrt{k}h)}=\sqrt{1-u^2}$ and $\dd u = \sqrt{k}\cos (\sqrt{k}h)$. This yields
\begin{align*}
I=
\frac{N_d}{N_{d-1}} \frac{2}{\sqrt{k}}\int_0^{1} \int_{0}^{\sqrt{1-z^2}} \tilde{f}\left(\frac{z}{\sqrt{1-u^2}}\right) (1-u^2)^\frac{d-2}{2} \, \dd u \, \tilde{\nu}_d(\dd z).
\end{align*}
We choose for the Abel-type integral inversion the function $\tilde{f}(x):=\mathbf{1}_{(0,a)}(x) (x^{-2}-a^{-2})^{-\frac{1}{2}}x^{d-2}$. Replacing $\tilde{f}$ gives
\begin{align*}
I=
\frac{N_d}{N_{d-1}} \frac{2}{\sqrt{k}}\int_0^{1} \int_{0}^{\sqrt{1-z^2}} \mathbf{1}_{(0,a)}\left(\frac{z}{\sqrt{1-u^2}}\right)
 ((1-u^2)z^{-2}-a^{-2})^{-\frac{1}{2}} 
 \frac{z^{d-2}}{(1-u^2)^\frac{d-2}{2}}  (1-u^2)^\frac{d-2}{2} \, \dd u \, \tilde{\nu}_d(\dd z).
\end{align*}
We incorporate the indicator function into integration bounds as follows:
\begin{align*}
I&=
\frac{N_d}{N_{d-1}} \frac{2}{\sqrt{k}}\int_0^{a} \int_{0}^{\sqrt{1-a^{-2}z^2}} 
 ((1-u^2)z^{-2}-a^{-2})^{-\frac{1}{2}} 
 z^{d-2}    \, \dd u \, \tilde{\nu}_d(\dd z).
\end{align*}
Observe that $((1-u^2)z^{-2}-a^{-2})^{-\frac{1}{2}} = z(1-u^2-a^{-2}z^2)^{-\frac{1}{2}}$.
Set $u:=\sqrt{1-a^{-2}z^2} \sin \theta$, then $\dd u =\sqrt{1-a^{-2}z^2} \cos \theta\dd \theta$, and 
$$
z(1-u^2-a^{-2}z^2)^{-\frac{1}{2}}=z (1-(1-a^{-2}z^2) \sin^2\theta -a^{-2}z^2)^{-\frac{1}{2}}  =z
(( 1-a^{-2}z^2) \cos^2 \theta)^{-\frac{1}{2}}.
$$
Putting everything together yields
\begin{align*}
I&=
\frac{N_d}{N_{d-1}} \frac{2}{\sqrt{k}}\int_0^{a}  z^{d-1} \int_{0}^{\frac{\pi}{2}} 
\frac{(( 1-a^{-2}z^2) \cos^2 \theta)^{\frac{1}{2}}}{(( 1-a^{-2}z^2) \cos^2 \theta)^{\frac{1}{2}}}
   \, \dd \theta \, \tilde{\nu}_d(\dd z)\\
   &=
\frac{N_d}{N_{d-1}} \frac{\pi}{\sqrt{k}}\int_0^{a}  z^{d-1}  \, \tilde{\nu}_d(\dd z).
\end{align*}
By the previous result, \eqref{eq:sectionform2} turns into
\begin{align}\label{eq:nd-12}
\int_0^{a} \frac{y^{d-2}}{\sqrt{y^{-2}-a^{-2}}} \, \tilde{\nu}_{d-1}(\dd y)=
\frac{N_d}{N_{d-1}} \frac{\pi}{\sqrt{k}}\int_0^{a}  z^{d-1}  \, \tilde{\nu}_d(\dd z).
\end{align}
Now consider $\tilde{\nu}_{d}((0,a))$:
\begin{align}\label{eq:last12}
\tilde{\nu}_{d}((0,a))&= \int_0^a \, \tilde{\nu}_{d}(\dd z)\nonumber  
 =  \int_0^{a} \left( -
\int_z^\infty {\frac{\dd }{\dd u}\left( \frac{1}{u^{d-1}}\right) }\, \dd u\right)
z^{d-1} 
\, \tilde{\nu}_{d}(\dd z)\nonumber  \\
&= \int_0^{\infty} \int_0^{u\wedge a} \frac{d-1}{u^d}  z^{d-1}  \, \dd u\, \tilde{\nu}_d(\dd z) \nonumber \\
&=\int_0^{\infty} \frac{d-1}{u^d} \int_0^{u\wedge a}  z^{d-1}  \, \dd u\, \tilde{\nu}_d(\dd z).
\end{align}
From this, using Equation \eqref{eq:nd-12}, we derive
\begin{align*}
 \int_0^{\infty} \frac{d-1}{u^d} \int_0^{u\wedge a}  z^{d-1}  \, \dd u\, \tilde{\nu}_d(\dd z)=
\frac{N_{d-1}}{N_d} \frac{\sqrt{k}}{\pi}\int_0^{\infty} \frac{d-1}{u^d}   \int_0^{{u\wedge a}} \frac{x^{d-2}}{\sqrt{x^{-2}-({u\wedge a})^{-2}}} \, \tilde{\nu}_{d-1}(\dd x) \, \dd u.
\end{align*}
Next, we split the integral from $0$ to $a$ and from $a$ to $\infty$. The resulting two expressions can be evaluated as follows:
\begin{align}\label{eq:last22}
\int_a^\infty \frac{d-1}{u^d}   \int_0^{a} \frac{x^{d-2}}{\sqrt{x^{-2}-a^{-2}}} \, \tilde{\nu}_{d-1}(\dd x)\, \dd u
=\frac{ 1}{a^{d-2}}   \int_0^{a} \frac{x^{d-1}}{\sqrt{a^{ 2}-x^{ 2}}} \, \tilde{\nu}_{d-1}(\dd x)
\end{align}
and 
\begin{align}\label{eq:last32}
\int_0^a \frac{d-1}{u^d}   \int_0^{{u }} \frac{x^{d-2}}{\sqrt{x^{-2}-u^{-2}}} \, \tilde{\nu}_{d-1}(\dd x) \, \dd u
=
(d-1)\int_0^a x^{d-1}   \int_x^{a} \frac{u^{-(d-1)}}{\sqrt{u^{2}-x^{2}}}\, \dd u \, \tilde{\nu}_{d-1}(\dd x).
\end{align}
Recall that $\nu((a,l_k])=\tilde{\nu}((0,\cos (\sqrt{ k}a)))$. Combining \eqref{eq:last12}, \eqref{eq:last22} and \eqref{eq:last32} gives
\begin{align}\label{eq:lastform}
\nu_d((a,l_k])=&
\frac{N_{d-1}}{N_d} \frac{\sqrt{k}}{\pi}
\left(
\frac{ 1}{\cos^{d-2} (\sqrt{ k}a)}   \int_0^{\cos (\sqrt{ k}a)}   \frac{x^{d-1}}{\sqrt{\cos^{ 2} (\sqrt{ k}a)-x^{ 2}}} \, \tilde{\nu}_{d-1}(\dd x) 
\right.\nonumber
\\
&\quad +
\left.
(d-1)\int_0^{\cos (\sqrt{ k}a)} x^{d-1}   \int_x^{\cos(\sqrt{k}a)} \frac{u^{-(d-1)}}{\sqrt{u^{2}-x^{2}}}\, \dd u \, \tilde{\nu}_{d-1}(\dd x) 
\right)\nonumber \\
=&\frac{N_{d-1}\sqrt{k} }{N_d\pi}  
\left(
\frac{ 1}{\cos^{d-2} (\sqrt{ k}a)}   \int_a^{l_k}   \frac{\cos^{d-1} (\sqrt{ k}y)}{\sqrt{\cos^{ 2} (\sqrt{ k}a)-\cos^{ 2} (\sqrt{ k}y)}} \,  \nu_{d-1}(\dd y) 
\right. \nonumber
 \\
&\quad +
\left.
(d-1)\int_a^{l_k} \cos^{d-1}   (\sqrt{ k}y) \int_{\cos (\sqrt{ k}y)}^{\cos(\sqrt{k}a)} \frac{u^{-(d-1)}}{\sqrt{u^{2}-\cos^{2} (\sqrt{ k}y)}}\, \dd u \, \nu_{d-1}(\dd y) \right).
\end{align}
Now we set $u=\frac{\cos (\sqrt{k} y)}{\cos \theta}$, {for $\theta\in(0,\arccos\left( \frac{\cos (\sqrt{k}y)}{\cos (\sqrt{k}a)}\right) )$}, so $\dd u = \frac{\cos (\sqrt{k} y)}{\cos \theta} \tan \theta\, \dd \theta$ and $u^2-\cos^2 (\sqrt{k} y)=\cos^2 (\sqrt{k} y)\tan^2 \theta $.
Then, 
\begin{align}\label{eq:sec2}
\notag &\int_{\cos (\sqrt{ k}y)}^{\cos(\sqrt{k}a)} \frac{u^{-(d-1)}}{\sqrt{u^{2}-\cos^{2} (\sqrt{ k}y)}}\, \dd u  \\
&=
\int_0^{\arccos \frac{\cos (\sqrt{ k}y)}{\cos (\sqrt{ k}a)}}
\left(\frac{\cos \theta}{\cos (\sqrt{ k}y)}\right)^{d-1} \frac{1}{\cos (\sqrt{ k}y)\tan\theta} \frac{\cos (\sqrt{ k}y)}{\cos \theta}\tan\theta\, \dd \theta.
\end{align}
Plugging this into \eqref{eq:lastform} completes the proof.
\end{proof}

Having carried ot the proofs, which we were not able to unify, we can present the following unified solution of Wicksell's corpuscle problem for all constant spaces using our convention \eqref{eq:CosConvention}. We will see in the next section that the classical Euclidean case with $k=0$ can be included as a limit as $k\downarrow 0$ or $k\uparrow 0$.

\begin{corollary}[Inversion formula for constant curvature spaces]\label{cor:inver}
Let $\MM_k^d$ be a space of constant curvature $k\neq 0$. Then, for any $a\in(0,l_k)$,  
\begin{align*}
\nu_d((a,l_k]) 
=&\frac{N_{d-1}\sqrt{k} }{N_d\pi}  
\left(
\frac{ 1}{\cos^{d-2} (\sqrt{ k}a)}   \int_a^{l_k}   \frac{\cos^{d-1} (\sqrt{ k}y)}{\sqrt{\cos^{ 2} (\sqrt{ k}a)-\cos^{ 2} (\sqrt{ k}y)}} \,  \nu_{d-1}(\dd y) 
\right.
 \\
&\quad +
\left.
(d-1)\int_a^{l_k}  \cos^{d-1}   (\sqrt{ k}y) \int_{\cos (\sqrt{ k}y)}^{\cos (\sqrt{k}a)}  \frac{{k}   u^{-(d-1)}}{ \sqrt{{k^2 } (  u^{2}-\cos^{2} (\sqrt{ k}y) }  ) }\, \dd u \,  \nu_{d-1}(\dd y) \right).
\end{align*}
\end{corollary}


\section{Convergence to the Euclidean case}\label{sec:asympt}

In this section, we show that the section formula and the inversion formula for $\MM_k^d$, a space of constant curvature $k\neq 0$, converge to the classical formulas in Wicksell's corpuscle problem in $\mathbb{R}^d$ if we let $k\to 0$. For the section formulas the convergence is illustrated in Figure \ref{fig:sectional_densities_convergence}.

\begin{figure}[t]
	\centering
	\includegraphics[width=0.32\columnwidth]{SectionalDensity.pdf}\ 
	\includegraphics[width=0.32\columnwidth]{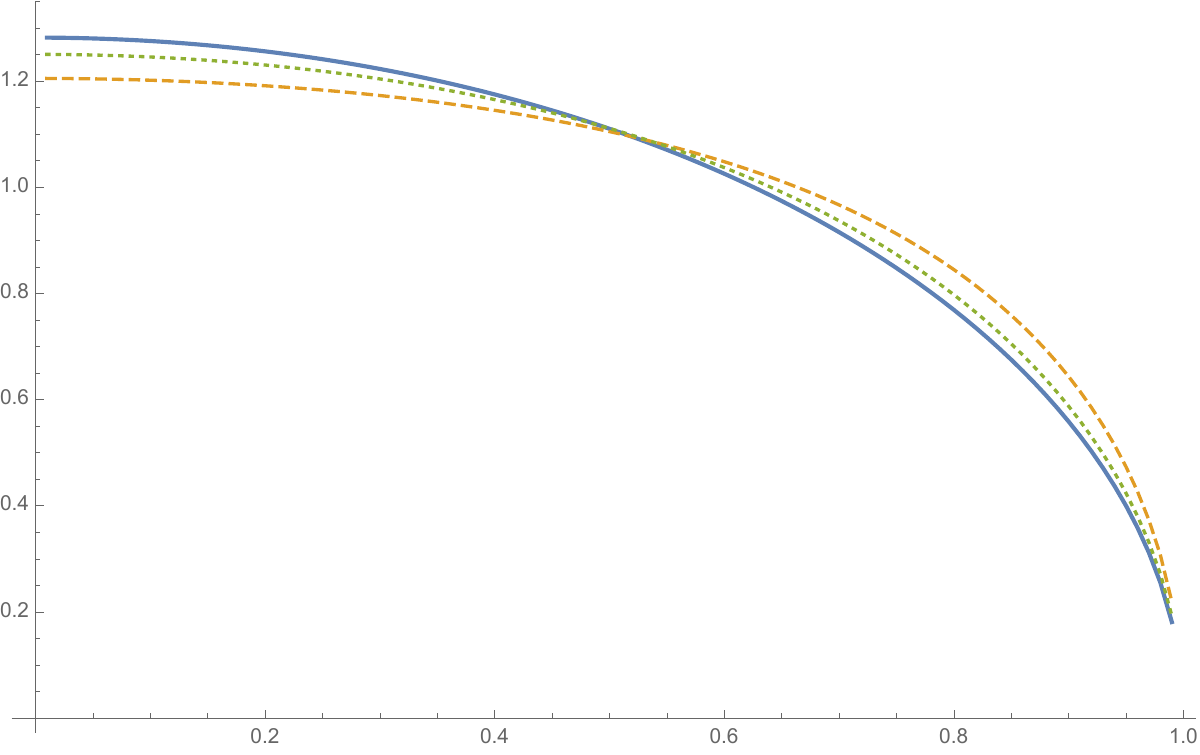}\ 
	\includegraphics[width=0.32\columnwidth]{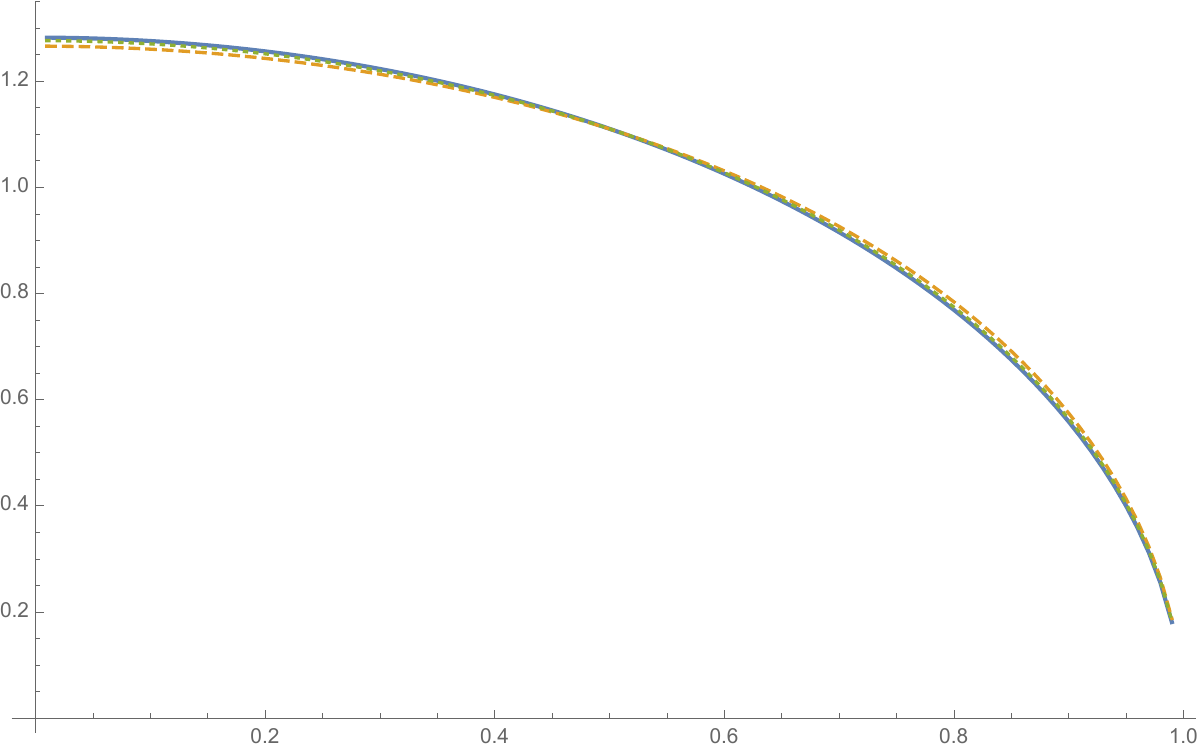}
	\caption{Densities of the measures $\nu_{d-1}$ of observed section radii when slicing a random process of balls of fixed radius \(1\) in $3$-dimensional spaces of constant curvature $k$. The solid blue curve corresponds to the choice $k=0$. The left panel shows the situation with $k=\pm 1$, the middle panel with $k=\pm 0.5$ and the left with $k=\pm 0.1$ with positive sign corresponding to dashed, orange and negative sign to dotted, green.}
	\label{fig:sectional_densities_convergence}
\end{figure}

\begin{theorem}
Let $\MM_k^d$ be a space of constant curvature $k{\neq 0}$, let $L$ be a totally geodesic hypersurface in $\MM_k^d$ and let $\nu_d$ be the distribution of the radius of the balls of the random ball process in $\MM_k^d$. Let $\nu_{d-1}^{(k)}$ be the induced radius distribution of the sectional balls within $L$. Then
$$
\nu_{d-1}^{(k)} \xrightarrow[k \to 0]{w} \nu_{d-1}^{(0)},
$$
i.e., $\nu_{d-1}^{(k)}$ converges weakly to $\nu_{d-1}^{(0)}$ as the curvature $k$ tends to $0$.
\end{theorem}
\begin{proof}  
We study the behaviour of our formulas as $k$ approaches $0$. Recall from Theorem \ref{th:ratioestimate} that the distribution of the radii within the section, for both positive and negative curvature, is given by
$$
   \nu^{(k)}_{d-1}((r,l_k]) 
= \frac{N^{(k)}_d}{N^{(k)}_{d-1}} \int_r^{l_k} \int_{-\alpha_k(R,r)}^{ \alpha_k(R,r)}  \cos^{d-1}(\sqrt{k}h)  \,\dd h\, \nu_{d}(\dd R),\qquad 0<r<l_k,
$$
where $l_k = \infty$ when $k < 0$.

We use, {for fixed $x\in\RR$}, the complex Taylor series at $0$ of the cosine function
$$
{\cos  (\sqrt{ k}x)=1-\frac{(\sqrt{ k}x)^2}{2}+O(k^2)} ,
$$
so that
$$\frac{\cos  (\sqrt{ k}x)}{\cos (\sqrt{ k}y)}= 1+k \frac{y^2-x^2}{2}+O(k^2)     {,\qquad  \text{as }k\to 0 } .$$ 
Since {$\mathrm{arcos} (1-x)= \sqrt{2x}-o(\sqrt{x})$ } as $x\to 0$, applying this to {$x=k \frac{r^2-R^2}{2}+O(k^2)$ } gives
$$
{\frac{1}{\sqrt{ k}}\mathrm{arcos }\left(\frac{\cos (\sqrt{ k} R)}{\cos  (\sqrt{ k}  r  )}\right) \to 
 \sqrt{R^2-r^2}, {\qquad  \text{as } k\to 0 }. }
$$ 

Next, we examine the ratio of the intensity parameters, as the constants in the section formula depend on the curvature. Recall that
\begin{align*}
\frac{N^{(k)}_{d-1}}{N^{(k)}_{d}}
&= \int_0^{\infty} \int_{-R}^{ R} {\cos ^{d-1}(\sqrt{ k}h) } \,\dd h\, \nu_{d}(\dd R) .
\end{align*}
This expression is decreasing in $k$, {for negative $k$, and it is bounded by its value at $k=0$ for non-negative $k$}.
Thus, by the dominated convergence theorem we obtain
\begin{align*}
\lim_{k\to 0} \frac{N^{(k)}_{d-1}}{N^{(k)}_{d}}
&= 2 \int_0^{\infty} R \, \nu_{d}(\dd R)=2\rho_d.
\end{align*}
Moreover, since $\alpha_k(R, r)$ increases as ${\vert k \vert } \downarrow 0$, and {$\cos (\sqrt{ k}h)$ decreases for negative $k$ and it is bounded by its value at $k=0$ for non-negative } (and thus is bounded by $\cosh(\sqrt{-k} \sqrt{R^2 - r^2})$), we obtain	
$$\int_{-\alpha_k(R,r)}^{ \alpha_k(R,r)} { \cos^{d-1}(\sqrt{k}h)  } \,\dd h  \to 2 \sqrt{R^2-r^2}  ,\qquad\text{as } k\to 0.
$$
Additionally, for {negative $k$} we have the upper bound
\begin{align*}
  \int_r^{\infty} \int_{-\alpha_k(R,r)}^{ \alpha_k(R,r)}  \cos^{d-1}(\sqrt{k}h)  \,\dd h\, \nu_{d}(\dd R)
 \leq 
 \int_0^{\infty} \int_{-R}^{ R}  \cos^{d-1}(\sqrt{k}h)  \,\dd h\, \nu_{d}(\dd R) 
\end{align*}
which in turn converges. {For non-negative  $k$ the integral is trivially bounded.} So, applying the dominated convergence theorem again yields
\begin{align*}
 \lim_{k\to 0 } \int_r^{\infty} \int_{-\alpha_k(R,r)}^{ \alpha_k(R,r)}  \cos^{d-1}(\sqrt{k}h)  \,\dd h\, \nu_{d}(\dd R)
& =
\int_r^{\infty}  \lim_{k\to 0 }  \int_{-\alpha_k(R,r)}^{ \alpha_k(R,r)}  \cos^{d-1}(\sqrt{k}h)  \,\dd h\, \nu_{d}(\dd R)\\
&=\int_r^{\infty}  2\sqrt{R^2-r^2} \, \nu_d(\dd R).
\end{align*}
This completes the proof.
\end{proof}

\begin{theorem}
Let $\MM_k^d$ be a space of constant curvature $k\neq 0$, let $L$ be a totally geodesic hypersurface in $\MM_k^d$ and let $\nu_{d-1} $ be the radius distribution of the sectional balls within $L$. Let $\nu_{d}^{(k)}$ be the distribution of balls in $\MM_k^d$ arising from unfolding the process in $L$. Then
$$
\nu_d^{(k)} \xrightarrow[k \to 0]{w} \nu_d^{(0)}.
$$
i.e., $\nu_d^{(k)}$ converges weakly to $\nu_d^{(0)}$ as the curvature $k$ tends to $0$ .
\end{theorem}
\begin{proof}
Recall from \cref{cor:inver} that the measure of the distribution of radii via the unfolding process is given by
\begin{align}\label{eq:1integrals}
\nu^{(k)}_d((a,l_k]) 
=&\frac{N_{d-1}\sqrt{k} }{N_d\pi}  
\left(
\frac{ 1}{\cos^{d-2} (\sqrt{ k}a)}   \int_a^{l_k}   \frac{\cos^{d-1} (\sqrt{ k}y)}{\sqrt{\cos^{ 2} (\sqrt{ k}a)-\cos^{ 2} (\sqrt{ k}y)}} \,  \nu_{d-1}(\dd y) 
\right. \nonumber 
 \\
&\quad +
\left.
(d-1)\int_a^{l_k}  \cos^{d-1}   (\sqrt{ k}y) \int_{\cos (\sqrt{ k}y)}^{\cos (\sqrt{k}a)}  \frac{{k}   u^{-(d-1)}}{ \sqrt{{ k^2 } (  u^{2}-\cos^{2} (\sqrt{ k}y) }  ) }\, \dd u \,  \nu_{d-1}(\dd y) \right).
\end{align}
where $a\in(0,l_k)$ and $l_k<\infty$ if $k>0$. 
We will show that the first integral converges to the inversion formula in the flat space $\mathbb{R}^d$, while the second term tends to zero as the curvature approaches zero.

We begin by recalling that $\cos (\sqrt{k}x)=1-\frac{(\sqrt{k}x)^2}{2}+O(k^2)$, and similarly, $\cos^n(\sqrt{k}x)=1-\frac{n}{2}(\sqrt{k} x)^2+O(k^2)$ for $n\in\N${, as $k\to 0$}. Therefore,
$$\sqrt{\cos^{ 2} (\sqrt{ k}a)-\cos^{ 2} (\sqrt{ k}y)}=\sqrt{(\sqrt{k}y)^2-(\sqrt{k}a)^2+O(k^2)}=\sqrt{k}\sqrt{y^2-a^2}(1+O(k))\quad \text{as } k {\to} 0$$ 
and thus
$$
\frac{ \sqrt{k}}{\cos^{d-2} (\sqrt{ k}a)}     \frac{\cos^{d-1} (\sqrt{ k}y)}{\sqrt{\cos^{ 2} (\sqrt{ k}a)-\cos^{ 2} (\sqrt{ k}y)}} = {\sqrt{k}\,(1+O(k))\over(1+O(k))\sqrt{k}\,\sqrt{y^2-a^2}\,(1+O(k))} =  \frac{1}{\sqrt{y^2-a^2}}(1+O(k)).
$$
{If $k$ is negative, then the aforementioned function is decreasing as $k$ increases to $0$. We can then apply the dominated convergence theorem. If $k$ is positive, we will use the following bound. We fix $a>0$ and take $k$ so small that $a<\pi/(6\sqrt{k})$. For $y\in(a,\pi/(4\sqrt{k}))$, $\frac{\cos^{d-1} (\sqrt{ k}y)}{\sqrt{\cos^{ 2} (\sqrt{ k}a)-\cos^{ 2} (\sqrt{ k}y)}} \leq \frac{1}{\sqrt{y^2-a^2}}$ and for $y\in(\pi/(4\sqrt{k}),l_k)$, $\frac{\cos^{d-1} (\sqrt{ k}y)}{\sqrt{\cos^{ 2} (\sqrt{ k}a)-\cos^{ 2} (\sqrt{ k}y)}} \leq C$ for some constant $C>0$.} By the dominated convergence theorem we obtain
$$
\frac{ \sqrt{k}}{\cos^{d-2} (\sqrt{ k}a)}   \int_0^{\infty}\mathbf{1}_{(a,l_k)}(y)   \frac{\cos^{d-1} (\sqrt{ k}y)}{\sqrt{\cos^{ 2} (\sqrt{ k}a)-\cos^{ 2} (\sqrt{ k}y)}} \,  \nu_{d-1}(\dd y)  \to 
\int_a^\infty \frac{1}{\sqrt{y^2-a^2}}\, \nu_{d-1}(\dd y),
$$
as $k\to 0$.
 This corresponds to the inversion formula for the corpuscle problem in $\mathbb{R}^d$, {see \cite{Wicksell1}}.

Next, we show that the second integral in \eqref{eq:1integrals} vanishes as the curvature tends to zero. {We will use the expressions from \cref{thm:Inversionk<0} and \cref{th:inveposi} separately. For $k<0$, we want to bound 
$$
(d-1)\int_a^{\infty}\int_{\arcsin \left(\frac{\cosh(\sqrt{-k}a)}{\cosh(\sqrt{-k}y)} \right)}^{\frac{\pi}{2} }
    ( \sin\theta)^{-(d-1)}  \, \dd \theta \, \nu_{d-1}(\dd y) =: (d-1)\int_a^{\infty} F_k(y) \, \nu_{d-1}(\dd y)  .
$$
Observe that $\arcsin \left(\frac{\cosh(\sqrt{-k}a)}{\cosh(\sqrt{-k}y)} \right) $ increases to $\frac{\pi}{2}$ as $k\uparrow0$  and thus $F_k(y)$ tends to $0$ as $k\uparrow0$. 
If $\nu_d^{(k)}$ is finite for some $k<0$, then we can apply the dominated convergence theorem and obtain that
$$
\lim_{k\to 0} (d-1)\int_a^{\infty}\int_{\arcsin \left(\frac{\cosh(\sqrt{-k}a)}{\cosh(\sqrt{-k}y)} \right)}^{\frac{\pi}{2} }
    ( \sin\theta)^{-(d-1)}  \, \dd \theta \, \nu_{d-1}(\dd y) = 0 .
$$
 }
{For $k>0$, we will bound 
$$(d-1)\int_a^{l_k} \int_0^{\arccos \frac{\cos (\sqrt{ k}y)}{\cos (\sqrt{ k}a)}}
\cos^{d-2}\theta\, \dd \theta \, \nu_{d-1}(\dd y).$$
Observe that }
$$
\int_0^{\arccos \frac{\cos(\sqrt{k}y)}{\cos(\sqrt{k}a)}} \cos^{d-2} \theta \, \dd \theta \leq \arccos \left( \frac{\cos(\sqrt{k}y)}{\cos(\sqrt{k}a)} \right),
$$
and that $\arccos x \leq \pi$. Since $\nu_{d-1}$ is a probability measure, for any $\epsilon > 0$, there exists $N_0 > 0$ such that $\nu_{d-1}((N, \infty)) < \epsilon$, for any $N>N_0$. Hence, 
\begin{align*}
 \int_a^{l_k} \int_0^{\arccos \frac{\cos(\sqrt{k}y)}{\cos(\sqrt{k}a)}} \cos^{d-2} \theta \, \dd \theta \, \nu_{d-1}(\dd y)
&\leq \int_a^{l_k} \arccos \left( \frac{\cos(\sqrt{k}y)}{\cos(\sqrt{k}a)} \right) \nu_{d-1}(\dd y) \\
&\leq \int_a^N \arccos \left( \frac{\cos(\sqrt{k}y)}{\cos(\sqrt{k}a)} \right) \nu_{d-1}(\dd y) + \epsilon \pi.
\end{align*}
 To complete the argument, observe that, since $y \leq N$, we have
$$
\arccos \left( \frac{\cos(\sqrt{k}y)}{\cos(\sqrt{k}a)} \right) \leq \arccos \left( \frac{\cos(\sqrt{k}N)}{\cos(\sqrt{k}a)} \right) \to 0 \quad \text{as } k \downarrow 0.
$$
The proof is is completed by taking $\epsilon\to 0$.
\end{proof}

\footnotesize{
\bibliographystyle{abbrv}
\bibliography{biblio}

\begin{thebibliography}{10}

\bibitem{Solodovnikov}
D.~V. Alekseevskij, E.~B. Vinberg, and A.~S. Solodovnikov.
\newblock {\em Geometry of Spaces of Constant Curvature}, pages 1--138.
\newblock Springer Berlin Heidelberg, Berlin, Heidelberg, 1993.

\bibitem{perc}
I.~Benjamini and O.~Schramm.
\newblock Percolation in the hyperbolic plane.
\newblock {\em Journal of the American Mathematical Society}, 14(2):487--507,
  2001.

\bibitem{Cyl1}
C.~Betken, E.~Broman, A.~Gusakova, and C.~Th{\"a}le.
\newblock Percolation of fat {Poisson} cylinders in hyperbolic space.
\newblock Preprint, {arXiv}:2412.17757, 2024.

\bibitem{Hyp1}
C.~Betken, D.~Hug, and C.~Th{\"a}le.
\newblock Intersections of {Poisson} {{\(k\)}}-flats in constant curvature
  spaces.
\newblock {\em Stochastic Processes Appl.}, 165:96--129, 2023.

\bibitem{Cyl2}
E.~Broman and J.~Tykesson.
\newblock {Poisson cylinders in hyperbolic space}.
\newblock {\em Electronic Journal of Probability}, 20:1 -- 25, 2015.

\bibitem{RCM1}
J.~Chellig, N.~Fountoulakis, and F.~Skerman.
\newblock The modularity of random graphs on the hyperbolic plane.
\newblock {\em Journal of Complex Networks}, 10(1):cnab051, 2021.

\bibitem{AchilleEtAl}
M.~D'Achille, N.~Curien, N.~Enriquez, R.~Lyons, and M.~{\"U}nel.
\newblock Ideal {Poisson}-{Voronoi} tessellations on hyperbolic spaces.
\newblock Preprint, {arXiv}:2303.16831, 2025.

\bibitem{RCM3}
M.~Dickson.
\newblock Non-uniqueness {Phase} in {Hyperbolic} {Marked} {Random} {Connection}
  {Models} using the {Spherical} {Transform}.
\newblock Preprint, {arXiv}:2412.12854, 2024.

\bibitem{RCM2}
M.~Dickson and M.~Heydenreich.
\newblock Mean-field behaviour of the random connection model on hyperbolic
  space.
\newblock Preprint, {arXiv}:2505.09025, 2025.

\bibitem{Dress1992}
H.~Dress and R.-D. Reiss.
\newblock Tail behavior in {W}icksell's corpuscle problem.
\newblock {\em Probability Theory and Applications: Essays to the Memory of
  J{\'o}zsef Mogyor{\'o}di}, pages 205--220, 1992.

\bibitem{RGG2}
N.~Fountoulakis, P.~van~der Hoorn, T.~M{\"u}ller, and M.~Schepers.
\newblock {Clustering in a hyperbolic model of complex networks}.
\newblock {\em Electronic Journal of Probability}, 26:1 -- 132, 2021.

\bibitem{RGG1}
N.~Fountoulakis and J.~Yukich.
\newblock Limit theory for isolated and extreme points in hyperbolic random
  geometric graphs.
\newblock {\em Electron. J. Probab.}, 25:51, 2020.

\bibitem{Gili2024}
F.~Gili, G.~Jongbloed, and A.~van~der Vaart~and.
\newblock Adaptive and efficient isotonic estimation in {W}icksell's problem.
\newblock {\em Journal of Nonparametric Statistics}, pages 1--41, 2024.

\bibitem{Groeneboom1995}
P.~Groeneboom and G.~Jongbloed.
\newblock {Isotonic Estimation and Rates of Convergence in Wicksell's Problem}.
\newblock {\em The Annals of Statistics}, 23(5):1518 -- 1542, 1995.

\bibitem{PV2}
B.~Hansen and T.~M{\"u}ller.
\newblock {Poisson–Voronoi percolation in the hyperbolic plane with small
  intensities}.
\newblock {\em The Annals of Probability}, 52(6):2342 -- 2405, 2024.

\bibitem{PV1}
B.~T. Hansen and T.~Müller.
\newblock The critical probability for {V}oronoi percolation in the hyperbolic
  plane tends to 1/2.
\newblock {\em Random Structures \& Algorithms}, 60(1):54--67, 2022.

\bibitem{Hyp2}
F.~Herold, D.~Hug, and C.~Th{\"a}le.
\newblock Does a central limit theorem hold for the k-skeleton of {P}oisson
  hyperplanes in hyperbolic space?
\newblock {\em Probability Theory and Related Fields}, 179(3):889--968, 2021.

\bibitem{Bool2}
D.~Hug, G.~Last, and M.~Schulte.
\newblock Boolean models in hyperbolic space.
\newblock Preprint, {arXiv}:2408.03890, 2024.

\bibitem{Hyp4}
Z.~Kabluchko, D.~Rosen, and C.~Th{\"a}le.
\newblock A quantitative central limit theorem for {Poisson} horospheres in
  high dimensions.
\newblock {\em Electron. Commun. Probab.}, 29:11, 2024.

\bibitem{Hyp3}
Z.~Kabluchko, D.~Rosen, and C.~Th{\"a}le.
\newblock Fluctuations of $\lambda$-geodesic poisson hyperplanes in hyperbolic
  space.
\newblock {\em Israel Journal of Mathematics}, 2025.

\bibitem{RGG3}
D.~Krioukov, F.~Papadopoulos, M.~Kitsak, A.~Vahdat, and M.~Boguñá.
\newblock Hyperbolic geometry of complex networks.
\newblock {\em Physical review. E, Statistical, nonlinear, and soft matter
  physics}, 82:036106, 2010.

\bibitem{Mecke1980}
J.~Mecke and D.~Stoyan.
\newblock Stereological problems for spherical particles.
\newblock {\em Mathematische Nachrichten}, 96:311--317, 1980.

\bibitem{Ohser2001}
J.~Ohser and F.~Mücklich.
\newblock Statistical analysis of microstructures in materials science.
\newblock {\em Practical Metallography}, 38(9):538--539, 2001.

\bibitem{Ratcc}
J.~G. Ratcliffe.
\newblock {\em Hyperbolic Geometry}.
\newblock Springer International Publishing, Cham, 2019.

\bibitem{Mecke1987}
D.~Stoyan, W.~S. Kendal, and J.~Mecke.
\newblock {\em Stochastic Geometry and Its Applications}.
\newblock De Gruyter, Berlin, Boston, 1987.

\bibitem{RSS}
E.~Sönmez, P.~Spanos, and C.~Thäle.
\newblock Intersection probabilities for flats in hyperbolic space.
\newblock {\em Advances in Mathematics}, 479:110415, 2025.

\bibitem{Bool3}
J.~Tykesson.
\newblock The number of unbounded components in the {P}oisson {B}oolean model
  of continuum percolation in hyperbolic space.
\newblock {\em Electron. J. Probab.}, 12:no. 51, 1379--1401, 2007.

\bibitem{Bool1}
J.~Tykesson and P.~Calka.
\newblock Asymptotics of visibility in the hyperbolic plane.
\newblock {\em Advances in Applied Probability}, 45(2):332–350, 2013.

\bibitem{Wicksell1}
S.~D. Wicksell.
\newblock The corpuscle problem: A mathematical study of a biometric problem.
\newblock {\em Biometrika}, 17(1/2):84--99, 1925.

\bibitem{Wicksell2}
S.~D. Wicksell.
\newblock The corpuscle problem: Second memoir: Case of ellipsoidal corpuscles.
\newblock {\em Biometrika}, 18(1/2):151--172, 1926.

\bibitem{zahle}
M.~Zähle.
\newblock Wicksell's corpuscle problem in spherical spaces.
\newblock {\em Journal of Applied Probability}, 27(3):701--706, 1990.

\end{thebibliography}
}

\end{document}